\documentclass[a4paper,11pt]{amsart}
\usepackage{hyperref,fancyhdr,mathrsfs,amsmath,amscd,amsthm,amsfonts,latexsym,amssymb,stmaryrd}
\usepackage[all]{xy}

\DeclareMathOperator{\Div}{div}
\DeclareMathOperator{\trace}{trace}

\DeclareMathOperator{\dif}{d}

\newcommand{\Cal}{\mathcal{C}}
\renewcommand{\H}{\mathscr{H}}

\newcommand{\Fa}{\mathcal{F}}

\def \a{\alpha}

\def \b{\beta}

\def \gf{\mathfrak{g}}

\def \G{\Gamma}

\def \o{\omega}

\def \phi{\varphi}
\def \Phi{\varPhi}
\def \p{\pi}
\def \r{\rho}
\def \s{\sigma}
\def \t{\tau}

\def \R{\mathbb{R}}
\def \Hq{\mathbb{H}\,}
\def \C{\mathbb{C}\,}

\def\widecheckg{g^{\hspace*{-2.5pt}\vbox to 5pt{\hbox to
0pt{\LARGE$\check{}$}}}\hspace*{2pt}}

\def\widecheckl{\lambda^{\hspace*{-3.5pt}\vbox to 8pt{\hbox to
0pt{\LARGE$\check{}$}}}\hspace*{2pt}}

\begin{document}

\title{Harmonic maps and twistorial structures}
\author{G.~Deschamps, E.~Loubeau, R.~Pantilie}  
\thanks{R.~Pantilie acknowledges partial financial support from the Romanian National Authority for Scientific Research, CNCS-UEFISCDI, 
project no.\ PN-III-P4-ID-PCE-2016-0019.} 
\address{G.~Deschamps, D\'epartement de Math\'ematiques, Universit\'e de Bretagne Occidentale, 6,
avenue Victor Le Gorgeu, CS 93837, 29238 Brest Cedex 3, France} 
\email{\href{mailto:Guillaume.Deschamps@univ-brest.fr}{Guillaume.Deschamps@univ-brest.fr}}
\address{E.~Loubeau, D\'epartement de Math\'ematiques, Universit\'e de Bretagne Occidentale, 6,
avenue Victor Le Gorgeu, CS 93837, 29238 Brest Cedex 3, France} 
\email{\href{mailto:Eric.Loubeau@univ-brest.fr}{Eric.Loubeau@univ-brest.fr}} 
\address{R.~Pantilie, Institutul de Matematic\u a ``Simion~Stoilow'' al Academiei Rom\^ane,
C.P. 1-764, 014700, Bucure\c sti, Rom\^ania} 
\email{\href{mailto:Radu.Pantilie@imar.ro}{Radu.Pantilie@imar.ro}} 

\newtheorem{thm}{Theorem}[section]
\newtheorem{lem}[thm]{Lemma}
\newtheorem{cor}[thm]{Corollary}
\newtheorem{prop}[thm]{Proposition}

\theoremstyle{definition}

\newtheorem{defn}[thm]{Definition}
\newtheorem{rem}[thm]{Remark}
\newtheorem{exm}[thm]{Example}

\numberwithin{equation}{section}

\maketitle
\thispagestyle{empty}
\vspace{-3mm}
\begin{center}
\emph{This paper is dedicated to the memory of \c Stefan Papadima.}
\end{center}
\vspace{1mm}
\begin{abstract}
We introduce the notion of Riemannian twistorial structure and we show that it provides new natural constructions of harmonic maps.   
\end{abstract}

\section*{Introduction}  

\indent 
Perhaps that the main apparently embarrassing problem in differential geometry is the lack of an obvious notion of morphism. We argue that this is 
not an well posed problem, that is, the definition of a `linear $G$-structure', for some Lie group $G$, as a faithful representation of $G$ is too raw:  
essentially, we have to consider only the representations preserving a holomorphic embedding of a compact complex manifold $Y$ into some complex Grassmannian of the representation's vector space $U$ (cf.\ \cite{Pan-rtwist}\,). Note that, this is not too restrictive as, assuming the 
complexified representation irreducible, we can, for example, take $Y$ to be the closed $G^{\C\!}$-orbit in the projectivisation of 
the complexification of $U^*$, setting which proved to be crucial in the classification of irreducible torsion free affine holonomies 
\cite{Mer}\,,\,\cite{MerSch}\,.\\ 
\indent 
Consequently, the `$G$-structures' should be given by suitable twistorial structures and the `morphisms' be just the corresponding twistorial maps 
of \cite{PanWoo-sd}\,,\,\cite{LouPan-II}\,.\\ 
\indent 
One of the main tasks of this paper is to construct the suitable definition of `Riemannian twistorial structure'. This is based on the 
`Euclidean (twistorial) $Y$-structures', that is, embeddings of $Y$ into the generalized Grassmannian of coisotropic subspaces 
(that is, orthogonal complements of isotropic subspaces) of the corresponding Euclidean space. As is usual in twistor theory, examples abound 
(see Section \ref{section:Euclid_Y}\,, below) with the more or less obvious particular cases of Hermitian (where $Y$ is a two points space) 
and metric quaternionic ($Y$ the Riemann sphere \cite{Pan-hqo}\,) geometries. Furthermore, as coisotropy means isotropic orthogonal complement, 
there is always an almost CR manifold underlying the construction of a Riemannian twistorial structure and this, by following an idea of 
\cite{EelSal} (see \cite{Raw-f}\,), can be shown to provide natural constructions of harmonic maps (see Section \ref{section:EelSal_CR}\,).\\ 
\indent 
This brings us to the other main task of this paper, namely, to find new natural constructions of harmonic maps. On one hand, this is in continuation of  
our investigation of the interplay of the properties of a map of being twistorial and a harmonic morphism  
(\,\cite{PanWoo-sd}\,,\,\cite{LouPan}\,,\,\,\cite{LouPan-II}\,) and, on the other, it provides the adequate level of generality 
for previously known twistorial constructions of harmonic maps (\,\cite{Bry-har_twistor}\,,\,\cite{BurRaw}\,). We achieve this just because,  
in our setting, \emph{any Riemannian symmetric space admits a nontrivial Riemannian twistorial structure invariant under the isometry group} 
(Theorem \ref{thm:sym_hts}\,) leading, for example, to the natural twistorial constructions of Example \ref{exm:hts_sym_(3)}\,,  
based on the new twistorial structure of Example \ref{exm:homog_hts}(3)\,.\\ 
\indent 
But there is also a feature of twistorial structures which seems to have not been observed before: the canonical projection from the `total space' 
onto the twistor space, endowed with natural Riemannian metrics, usually, pulls back holomorphic functions to harmonic functions and, consequently 
(see \cite{Lou-pseudo_hm}\,,\,\cite{ApApBr}\,,\,\cite{BaiWoo2} for more information on such maps), 
if the twistor space is K\"ahler (as it is the case of the main examples provided by generalized flag manifolds) this projection is a harmonic map. 
This gives Theorem \ref{thm:Rts} and Corollary \ref{cor:Rts} with the abundant sources of examples provided by Examples \ref{exm:flat_hts} 
and \ref{exm:homog_hts}\,.

\section{Euclidean twistorial structures} \label{section:Euclid_Y}

\indent 
The following definition specializes a notion of \cite{Pan-rtwist} (cf.\ \cite{Pan-hqo}\,). 

\begin{defn} \label{defn:Ets} 
Let $U$ be an Euclidean space, $\dim U=n$\,, let $k\in\mathbb{N}$\,, $n/2\leq k\leq n$\,, and let $Y$ be a compact complex manifold endowed 
with a conjugation; that is, an involutive antiholomorphic diffeomorphism.\\ 
\indent 
\emph{An Euclidean (twistorial) $Y$-structure on $U$} is a a holomorphic embedding $\s$ of $Y$ into
${\rm Gr}_k\bigl(U^{\C\!}\bigr)$ such that:\\ 
\indent 
\quad(1) $\s$ intertwines the conjugations;\\ 
\indent 
\quad(2) $\s(y)\subseteq U^{\C\!}$ is coisotropic, for any $y\in Y$;\\   
\indent 
We say that $\s$ is \emph{full/maximal} if, further, the following condition is satisfied:\\  
\indent 
\quad(3) The tautological exact sequence of holomorphic vector bundles over $Y$  
\begin{equation} \label{e:Ets} 
0\longrightarrow\bigsqcup_{y\in Y}\s(y)\longrightarrow Y\times U^{\C\!}\longrightarrow\mathcal{U}\longrightarrow0 
\end{equation}   
induces a linear map from $U^{\C\!}$ to $H^0(\mathcal{U})$ which is injective/bijective.\\ 
\indent 
A \emph{Hermitian $Y$-structure on $U$} is an Euclidean $Y$-structure on $U$ with $k=n/2$\,.\\ 
\indent 
If $\s_j$ are Euclidean $Y_j$-structures on $U_j$\,, $(j=1,2)$\,, \emph{a morphism from $(U_1,\s_1)$ to $(U_2,\s_2)$} is a pair $(A,\a)$\,, 
where $A:U_1\to U_2$ is linear, $\a:Y_1\to Y_2$ is holomorphic and intertwines the conjugations, and 
$A(\s_1(y))\subseteq\s_2(\a(y))$\,, for any $y\in Y_1$\,.\\ 
\indent 
If $\s$ is an Euclidean $Y$-structure on $U$, an \emph{automorphism of $(U,\s)$} is an isomorphism $(A,\a):(U,\s)\to(U,\s)$ such that $A\in{\rm SO}(U)$\,.  
\end{defn}  

\begin{rem} 
If $U$ is endowed with an Euclidean $Y$-structure $\s$ then $Y$ is K\"ahler. Also, any automorphism $(A,\a)$ of $(U,\s)$ 
preserves the tautological exact sequence \eqref{e:Ets}\,, and $\a$ is an isometry of $Y$. 
\end{rem} 

\begin{exm} \label{exm:Ets} 
1) In Definition \ref{defn:Ets}\,, if $k=n$ (and $Y$ is formed of just one point) then we obtain the \emph{trivial Euclidean structure} of $U$.\\ 
\indent 
2) Any Euclidean $Y$-structure is the product of a trivial Euclidean structure and a full Euclidean $Y$-structure. Indeed, if $\s$ is an Euclidean 
$Y$-structure on $U$ then $\bigcap_{y\in Y}\s(y)=V^{\C\!}$ for some vector subspace $V$ of $U$ and $\s(y)=V^{\C\!}\oplus\t(y)$\,, for any $y\in Y$, 
where $\t$ is a full Euclidean $Y$-structure on $V^{\perp}(\subseteq U)$.\\ 
\indent  
3) Let $n=\dim U$ be even and $Y$ be formed of just two points. Then any Hermitian $Y$-structure on $U$ is given by an orthogonal complex structure on $U$.\\ 
\indent 
4) Let $U=\Hq^{\!m}$ and let $S^2$ be the sphere of unit imaginary quaternions. By associating to any $q\in S^2$ the $-{\rm i}$ the eigenspace 
of the linear complex structure given by quaternionic multiplication to the left with $q$\,, we obtain a maximal Hermitian $S^2$-structure whose 
automorphism group is ${\rm Sp}(1)\cdot{\rm Sp}(m)$ embedded into $\bigl({\rm Sp}(1)\cdot{\rm Sp}(m)\bigr)\times{\rm SO}(3)$ 
as the graph of the obvious Lie groups morphism from ${\rm Sp}(1)\cdot{\rm Sp}(m)$ to ${\rm SO}(3)$\,.\\ 
\indent 
Moreover, any maximal Hermitian $S^2$-structure is obtained this way.\\ 
\indent 
5) Let $U=\R^3\times\R^m$ and let $S^2$ be embedded as the (isotropic) conic in the projectivisation of the complexification of the Euclidean space $\R^3$.   
By defining $\s(\ell)=\ell^{\perp}\otimes\C^{\!m}$, for any $\ell\in S^2\bigl(\subseteq\C\!P^2\bigr)$\,, we obtain a maximal Euclidean $S^2$-structure 
whose automorphism group is ${\rm SO}(3)\times{\rm SO}(m)$ (see \cite{fq_2} for more details on these structures).\\ 
\indent 
More generally, if $U$ is endowed with a linear $Y$-structure and $V$ is endowed with a (trivial) Euclidean structure then $U\otimes V$ is endowed 
with an Euclidean $Y$-structure whose automorphism group is covered by the direct product of the automorphism group of $U$ and ${\rm SO}(V)$\,.\\ 
\indent 
6) Let $G$ be a complex semisimple Lie group and $P\subseteq G$ a parabolic subgroup such that $G$ acts effectively on $Y=G/P$. 
Let $G_{\R}$ be a compact real form of $G$ and denote by $\mathfrak{g}$\,, $\mathfrak{g}_{\R}$\,, $\mathfrak{p}$ the Lie algebras 
of $G$, $G_{\R}$\,, $P$, respectively.\\ 
\indent 
As $\mathfrak{p}$ is its own normalizer in $\mathfrak{g}$ and $P$ is connected, we have a holomorphic embedding $\s:Y\to{\rm Gr}_k(\mathfrak{g})$ 
given by $\s(aP)=({\rm Ad}\,a)(\mathfrak{p})$\,, for any $a\in G$. Consequently, $\s$ is an Euclidean $Y$-structure on $\mathfrak{g}_{\R}$\,, 
where $\mathfrak{g}_{\R}$ is endowed with the opposite of its Killing form. Moreover, this is maximal as the corresponding tautological exact sequence 
\eqref{e:Ets} is the canonical exact sequence \cite{At-57} associated to the holomorphic principal bundle $(G,Y,P)$\,, and by \cite[Proposition II]{Som-73}\,, 
we have $\mathfrak{g}=H^0\bigl(T^{1,0}Y\bigr)$\,.\\ 
\indent 
Furthermore, the automorphism group of $\s$ is a semidirect product $H\rtimes G_{\R}$\,, where the complexification of the Lie algebra of $H$ is 
$H^0\bigl({\rm End}\bigl(T^{1,0}Y\bigr)\bigr)$\,.   
\end{exm} 

\begin{defn} \label{defn:induced_Y-str} 
Let $\r:E\to U$ be an orthogonal projection between Euclidean spaces. Suppose that $E$ is endowed with an Euclidean $Y$-structure 
$\s:Y\to{\rm Gr}\bigl(E^{\C\!}\bigr)$ such that $\s(y)\cap({\rm ker}\r)^{\C\!}=\{0\}$\,, for any $y\in Y$.\\ 
\indent 
Then $Y\to{\rm Gr}\bigl(U^{\C\!}\bigr)$\,, $y\mapsto\r^{\C\!}\bigl(\s(y)\bigr)$\,, $(y\in Y)$\,, is the Euclidean $Y$-structure on $U$ 
\emph{induced by $\r$ and $\s$}. 
\end{defn} 

\begin{lem} \label{lem:isotropic_part_of_induced} 
Let $E$ be endowed with an Euclidean $Y$-structure $\s$ and let $\r:E\to U$ be an orthogonal projection which induces a $Y$-structure on $U$.\\ 
\indent 
Then, for any $y\in Y$, we have $\r\bigl(\s(y)\bigr)^{\perp}=\s(y)^{\perp}\cap\,U^{\C\!}$. Consequently, for any $y\in Y$, the orthogonal projection 
of $\r\bigl(\s(y)^{\perp}\bigr)$ onto $\bigl(\r\bigl(\s(y)\bigr)\cap\overline{\r\bigl(\s(y)\bigr)}\,\bigr)^{\perp}$ is equal to $\r\bigl(\s(y)\bigr)^{\perp}$. 
\end{lem} 
\begin{proof} 
Let $y\in Y$ and denote $C=\s(y)$\,. As $U$ is the orthogonal complement in $E$ of ${\rm ker}\r$\,, the orthogonal complement in $U^{\C\!}$ 
of $\r(C)$ is equal to the orthogonal complement in $E^{\C\!}$ of $\r(C)+({\rm ker}\r)^{\C\!}$. Thus, we have to prove that 
$\bigl(\r(C)+({\rm ker}\r)^{\C\!}\bigr)^{\perp}=C^{\perp}\cap U^{\C}$; equivalently, $\r(C)+({\rm ker}\r)^{\C\!}=C+({\rm ker}\r)^{\C\!}$ 
which holds because both sides are equal to $\r^{-1}(\r(C))$\,.\\ 
\indent 
As $\r(C)^{\perp}=C^{\perp}\cap U^{\C\!}$, we have $\r(C)^{\perp}\subseteq\r\bigl(C^{\perp}\bigr)$ and the last statement follows.  
\end{proof} 

\begin{rem} 
Any maximal (quaternionic-like) Euclidean $S^2$-structure (in particular, the one given by Example \ref{exm:Ets}(5)\,) is induced by a 
Hermitian $S^2$-structure (as given by Example \ref{exm:Ets}(4)\,); moreover, maximality is a necessary condition for this to hold (see \cite{Pan-hqo}\,).\\ 
\indent 
But, not all maximal Euclidean $Y$-structures are induced by Hermitian $Y$-structures. For an example, let $Y$ be the hyperquadric of isotropic 
directions in the complexification of an Euclidean space $U$. If $\dim U\geq4$ then $\s:Y\to P\bigl(U^{\C}\bigr)$\,, $\s(y)=y^{\perp}$, $(y\in Y)$\,, 
is a maximal Euclidean $Y$-structure which cannot be induced by a Hermitian one.\\ 
\indent 
Indeed, let $0\longrightarrow\mathcal{E}\longrightarrow Y\times U^{\C\!}\longrightarrow\mathcal{U}\longrightarrow0$ 
be the corresponding tautological exact sequence \eqref{e:Ets} of holomorphic vector bundles. From this exact sequence we deduce 
that $H^0(\mathcal{E})=H^1(\mathcal{E})=\{0\}$\,, and from its dual, because $H^0(\mathcal{U}^*)=H^1(\mathcal{U}^*)=\{0\}$ 
(by the Kodaira vanishing theorem and Serre duality), we deduce $U^{\C\!}=H^0(\mathcal{E}^*)$\,.\\ 
\indent 
Now, if $\s$ would be induced by a Hermitian $Y$-structure on some Euclidean space $E$ then its tautological exact sequence would be 
$0\longrightarrow\mathcal{E}\longrightarrow Y\times E^{\C\!}\longrightarrow\mathcal{E}^*\longrightarrow0$ which would imply  
$E^{\C\!}=H^0(\mathcal{E}^*)=U^{\C\!}$, a contradiction.\\ 
\indent 
More generally, by taking $Y$ to be any generalized flag manifold, $\dim_{\C\!}Y\geq2$\,, endowed with suitable very ample line bundles,  
we obtain maximal Euclidean $Y$-structures not induced by Hermitian $Y$-structures.\\ 
\indent 
Similarly, it can be proved that the Euclidean $Y$-structure of Example \ref{exm:Ets}(6) is induced by a Hermitian $Y$-structure if and only if 
$Y$ is a product of Riemann spheres.\\ 
\indent 
Finally, let $E$ be endowed with an Euclidean $Y$-structure $\s$. Then an orthogonal projection $\r:E\to U$ induces a maximal Hermitian 
$Y$-structure on $U$ if and only if it is an isometry.   
\end{rem}

\section{Harmonic maps and CR twistor spaces} \label{section:EelSal_CR}

\indent 
Let $M$ be a (smooth, oriented) manifold, and let $(Y,\Cal)$ be an almost CR manifold; that is, $\Cal\subseteq T^{\C\!}Y$ such that $\Cal\cap\,\overline{\Cal}=0$\,. 
Then $(Y,\Cal)$ is an \emph{almost CR twistor space} (cf.\ \cite{LeB-CR}\,, \cite{LouPan-II}\,, \cite{fq}\,) of $M$ if there is a proper surjective 
submersion $\p:Y\to M$ such that:\\ 
\indent 
\quad(1) For any fibre $Y_x=\p^{-1}(x)$\,, $(x\in M)$\,, the intersection $\Cal\cap T^{\C\!}Y_x$ is the antiholomorphic tangent bundle of a complex structure 
on $Y_x$\,;\\ 
\indent 
\quad(2)  For any $x\in M$, the map $Y_x\to{\rm Gr}\bigl(T_x^{\C\!}M\bigr)$\,, $y\mapsto\dif\!\p\bigl(\Cal_y\bigr)$\,, $(y\in Y_x)$\,, is a holomorphic embedding.\\ 
\indent 
Suppose, now, that $M$ is endowed with a Riemannian metric. Then the almost CR twistor space $(Y,\Cal)$ is \emph{Riemannian} if $Y$ 
is endowed with a Riemannian metric such that the following conditions are, also, satisfied:\\ 
\indent 
\quad(3) $\p$ is a Riemannian submersion with totally geodesic fibres;\\ 
\indent 
\quad(4) $\Cal$ is isotropic;\\ 
\indent 
\quad(5) For any $x\in M$, the Riemannian metric induced from $Y$ on $Y_x$ is the same with the K\"ahler metric induced from 
${\rm Gr}\bigl(T_x^{\C\!}M\bigr)$\,;\\ 
\indent 
\quad(6) Under the obvious embedding $Y\subseteq{\rm Gr}\bigl(T^{\C\!}M\bigr)$\,, the horizontal distribution $\H=({\rm ker}\dif\!\p)^{\perp}$ on $Y$ 
is induced by the Levi-Civita connection of $M$. 

\begin{rem} \label{rem:CR_holo_constr} 
Let $x_0\in M$ and let $G$ be the automorphism group of the Euclidean $Y_{x_0}$-structure given by 
$Y_{x_0}\to{\rm Gr}\bigl(T_{x_0}^{\C\!}M\bigr)$\,, $y\mapsto\dif\!\p\bigl(\Cal_y\bigr)^{\perp}$\,, $\bigl(y\in Y_{x_0}\bigr)$\,.\\ 
\indent 
Then $G$ contains the holonomy group of $M$ at $x_0$ and, consequently, we may retrieve $(Y,\Cal,\p)$ from the $Y_{x_0}$-structure 
on $T_{x_0}M$ and the holonomy bundle of the orthonormal frame bundle of $M$. 
\end{rem} 

\indent 
Let $(Y,\Cal)$ be a Riemannian almost CR twistor space given by $\p:Y\to M$ and let $\H=({\rm ker}\dif\!\p)^{\perp}$. 
Then $\widetilde{\Cal}=\bigl(\Cal\cap\H^{\C\!}\bigr)\oplus({\rm ker}\dif\!\p)^{1,0}$ is an isotropic almost CR structure on $Y$. 
Essentially, the following result was obtained in \cite[\S5]{Raw-f} (cf.\ \cite{EelSal}\,).  

\begin{thm} \label{thm:CR_EelSal} 
Let $M$ be a Riemannian manifold endowed with a Riemannian almost CR twistor space $(Y,\Cal)$ given by $\p:Y\to M$. 
Let $(N,J)$ be an almost complex manifold and let $\phi:N\to Y$ be an immersion such that:\\ 
\indent 
\quad{\rm (i)} $\dif\!\phi(TN)\cap({\rm ker}\dif\!\p)=0$\,;\\ 
\indent 
\quad{\rm (ii)} $\phi:(N,J)\to(Y,\widetilde{\Cal})$ is a CR map (that is, $\dif\!\phi\bigl(T^{0,1}N\bigr)\subseteq\widetilde{\Cal}$).\\ 
\indent 
Then $\p\circ\phi:(N,J)\to M$ is a $(1,1)$-geodesic immersion and, in particular, minimal. Moreover, the metric on $N$ pulled back by $\p\circ\phi$ 
from $M$ is $(1,2)$-symplectic, with respect to $J$.  
\end{thm} 
\begin{proof} 
We may suppose that the image of $\phi$ is the image of a section of $Y$ over $N\,(\subseteq M)$\,. 
It is, also, convenient (see \cite{Raw-f}\,) to associate to $\Cal$ the almost $f$-structure $\Fa$ on $Y$ (that is, $\Fa$ is a section of ${\rm End}(TY)$ 
such that $\Fa^3+\Fa=0$\,) which is skew-adjoint and ${\rm ker}(\Fa+{\rm i})=\Cal$\,.\\ 
\indent 
Accordingly, $Y$ is a bundle of skew-adjoint linear $f$-structures on $M$ and therefore the section of $Y$ giving $\phi$ corresponds to a section 
$F$ of ${\rm End}(TM|_N)$\,. Furthermore, standard arguments show that condition (ii) is equivalent to the following:\\ 
\indent 
\quad(iia) $TN\subseteq{\rm ker}(F^2+1)$ and $F|_{TN}=J$\,;\\ 
\indent 
\quad(iib) $\nabla_{FX}F=(\nabla_XF)F$, for any $X\in TN$, where $\nabla$ is the Levi-Civita connection of $M$.\\ 
\indent 
Consequently, for any sections $X$ and $Y$ of $T^{1,0}N$, $\nabla_{\overline{X}}Y$ is a section of ${\rm ker}(F-{\rm i})$\,; equivalently, 
$(N,J)$ endowed with the metric induces from $M$ is $(1,2)$-symplectic, and  
$b(\overline{X},Y)\in{\rm ker}(F-{\rm i})$\,, where $b$ is the second fundamental form of $N\subseteq M$. Together with the fact 
that $b$ is symmetric we deduce that $b(\overline{X},Y)=0$\,, for any $X,Y\in T^{1,0}N$. 
\end{proof} 

\begin{rem} \label{rem:CR_EelSal_(0,2)} 
In (ii) of Theorem \ref{thm:CR_EelSal}\,, if we replace $\widetilde{\Cal}$ by $\Cal$ then similarly we obtain (essentially, \cite[\S5]{Raw-f}\,) 
that $(N,J)$ is integrable. 
\end{rem} 

\indent 
Essentially, the following result was obtained in \cite[\S5]{Raw-f}\,.   

\begin{cor} \label{cor:CR_EelSal_1} 
Let $M$ be a Riemannian manifold endowed with a Riemannian almost CR twistor space $(Y,\Cal)$ given by $\p:Y\to M$. 
Let $(N,J)$ be an almost complex manifold and let $\phi:N\to Y$ be an immersion such that:\\ 
\indent 
\quad{\rm (i)} $\phi(N)$ is orthogonal to the fibres of $\p$\,;\\ 
\indent 
\quad{\rm (ii)} $\phi:(N,J)\to(Y,\Cal)$ is a CR map.\\ 
\indent 
Then the immersion $\p\circ\phi:(N,J)\to M$ is $(1,1)$-geodesic and, in particular, minimal. 
Moreover, the metric on $N$ pulled back by $\p\circ\phi$ from $M$ is K\"ahler, with respect to $J$, and, thus, $\phi$ is pluriharmonic.   
\end{cor} 
\begin{proof} 
This is an immediate consequence of Theorem \ref{thm:CR_EelSal} and Remark \ref{rem:CR_EelSal_(0,2)}\,. 
\end{proof} 

\indent 
In certain cases, the converse of Theorem \ref{thm:CR_EelSal} (and, similarly, of Remark \ref{rem:CR_EelSal_(0,2)}\,) also holds. 
For example, let $M$ be a Riemannian manifold, let $k\in\mathbb{N}$\,, 
$2k\leq\dim M$, and let $Y$ be the bundle of isotropic subspaces of dimension $k$ on $T^{\C\!}M$. 
Then \cite{Pan-tm} (see Remark \ref{rem:CR_holo_constr}\,; cf.\  \cite{LeB-CR}\,, \cite{Ro-LeB_nonrealiz}\,) 
$Y$ is endowed with a \emph{canonical} almost CR structure $\Cal$ which if $k=1$ is integrable \cite{Ro-LeB_nonrealiz} (cf.\ \cite{LeB-CR}\,), 
and if $k\geq2$ then $\Cal$ is integrable if and only if $M$ is conformally flat \cite{Pan-tm}\,. Obviously, $(Y,\Cal)$ is a Riemannian almost CR twistor space.\\ 
\indent  
Essentially, the following result was obtained in \cite[\S5]{Raw-f} (see, also, \cite{Sal-har_holo_maps}\,, for the case $k=1$\,). 

\begin{cor} \label{cor:CR_EelSal_2} 
Let $M$ be a Riemannian manifold, let $k\in\mathbb{N}$\,, $2\leq2k\leq\dim M$, and let $Y$ be the bundle of isotropic subspaces of dimension 
$k$ on $T^{\C\!}M$ endowed with its canonical almost CR structure $\Cal$\,.\\ 
\indent  
Let $(N,J)$ be a $(1,2)$-symplectic almost Hermitian manifold, $\dim N=2k$\,, and let $\phi:N\to M$ be an isometric immersion. 
Then the following assertions are equivalent:\\ 
\indent 
\quad{\rm (i)} $\phi$ is a $(1,1)$-geodesic map;\\ 
\indent 
\quad{\rm (ii)} $(N,J)\to(Y,\widetilde{\Cal}\,)$\,, $x\mapsto\dif\!\phi\bigl(T_x^{0,1}N\bigr)$\,, $(x\in N)$\,, is a CR map. 
\end{cor} 
\begin{proof} 
This is an immediate consequence of Theorem \ref{thm:CR_EelSal}\,. 
\end{proof} 

\begin{exm} 
Let $k,n\in\mathbb{N}$\,, $2\leq2k\leq n$\,, let $(N,J)$ be a $(1,2)$-symplectic almost Hermitian manifold, $\dim N=2k$\,, 
and let $\phi:N\to\R^n$ be an isometric immersion. Denote by ${\rm Gr}_k^0(n,\C\!)$ the space of isotropic complex vector subspaces of dimension $k$ 
of~$\C^{\!n}$.\\ 
\indent 
Then $\phi$ is $(1,1)$-geodesic (and, in particular, harmonic) if an only if it induces an antiholomorphic map from $(N,J)$ to ${\rm Gr}_k^0(n,\C\!)$ 
(the case $k=1$ is classical, see \cite{Sal-har_holo_maps} and the references therein). This quickly gives the Weierstrass representation \cite{ArPiSo} 
of pluriharmonic immersions from K\"ahler manifolds to Euclidean spaces.  
\end{exm}    

\indent 
Some of the results of this section admit more general formulations which, however, require the following definitions. 

\begin{defn}
Let $(M,g)$ and $(N,h)$ be Riemannian manifolds, $(N,h)$ equipped with a skew-adjoint almost $f$-structure $F$. 
We say that a map $\phi : (N,h,F) \to (M,g)$ is
\begin{itemize}
\item[i)] pseudo-harmonic if
$$
\trace\left( \nabla d\phi \Big|_{H\times H} \right) =0 ,
$$
where $H = \ker( F^2 + \mathrm{Id} )$.
\item[ii)] pseudo-(1,1)-geodesic if
$$
\nabla d\phi (\overline{Z},Z) =0 ,
$$
for all $Z\in \ker (F-{\rm i})$.
\end{itemize}
\end{defn}

\begin{defn}
Let $(N,h)$ be a Riemannian manifold, equipped with a skew-adjoint almost $f$-structure $F$.
We say that 
\begin{itemize}
\item[i)] $(N,h,F)$ is CR-cosymplectic if the restriction of $\nabla^{N} J$ to $\ker (F + {\rm i}) \otimes \ker (F-{\rm i})$ is trace-free (with respect to the metric $h$), i.e.
$$
\sum_{i} \nabla_{\overline{Z_i}} Z_i \in \ker (F-{\rm i}) ,
$$
where $\{ Z_{i}\}$ is an orthonormal basis of $\ker (F-{\rm i})$.
\item[ii)] $(N,h,F)$ is CR-$(1,2)$-symplectic if $\nabla^{N} J$ maps $\ker (F + {\rm i}) \otimes \ker (F-{\rm i})$ into $\ker (F-{\rm i})$.
\end{itemize}
\end{defn} 

\begin{cor} \label{cor:Raw_5.6} 
Let $M$ be a Riemannian manifold endowed with a Riemannian almost CR twistor space $(Y,\Cal)$ given by $\p:Y\to M$.\\ 
\indent  
Let $N$ be a Riemannian manifold endowed with a skew-adjoint almost $f$-structure $J$ such that $(N,h,J)$ is a CR-cosymplectic Riemannian manifold.\\ 
\indent 
If $\phi:\bigl(N,{\rm ker}(J+{\rm i})\bigr)\to\bigl(Y,\widetilde{\Cal}\,\bigr)$ is a CR map then $\p\circ\phi$ is pseudo-harmonic. 
\end{cor} 
\begin{proof} 
Similarly to the proof of Theorem \ref{thm:CR_EelSal}\,, the map $\phi$ corresponds to a section $F$ of 
$(\p\circ\phi)^*(Y)\subseteq{\rm End}\bigl((\p\circ\phi)^*(TM)\bigr)$\,. Also, as $\phi:\bigl(N,{\rm ker}(J+{\rm i})\bigr)\to\bigl(Y,\widetilde{\Cal}\,\bigr)$ 
is a CR map, we have $\bigl((\p\circ\phi)^*\bigl(\nabla^M\bigr)\bigr)_{JX}F=\bigl(\bigl((\p\circ\phi)^*\bigl(\nabla^M\bigr)\bigr)_XF\bigr)F$, 
for any $X\in{\rm ker}(J^2+1)$\,, where $\nabla^M$ is the Levi-Civita connection of $M$.\\ 
\indent  
Thus, $\bigl((\p\circ\phi)^*\bigl(\nabla^M\bigr)\bigr)_{\overline{X}}\bigl(\dif(\p\circ\phi)(Y)\bigr)$ is a section of ${\rm ker}(F-{\rm i})$ 
for any sections $X$ and $Y$ of ${\rm ker}(J-{\rm i})$\,, where $\dif(\p\circ\phi)$ is seen as a section of ${\rm Hom}\bigl(TN,(\p\circ\phi)^*(TM)\bigr)$\,.\\ 
\indent  
Consequently, the trace of $\nabla\!\dif(\p\circ\phi)$ restricted to ${\rm ker}(J^2+1)$ takes values in ${\rm ker}(F-{\rm i})$ and 
the proof quickly follows. 
\end{proof} 

\indent 
Corollary \ref{cor:Raw_5.6} gives \cite[Theorem 5.6]{Raw-f} if $J$ is an almost Hermitian structure. 

\begin{cor}
Let $(M,g)$ be a Riemannian manifold, $k\in \mathbb N$, $2\leq 2k\leq \dim M$, and $Y$ the bundle of isotropic subspaces of dimension $k$ on $T^\C M$  endowed with its canonical CR structure $\tilde{\mathcal C}$. 
Let $(N,h)$ be a Riemannian manifold endowed with a skew-adjoint almost $f$-structure $F$ of rank $2k$ and such that $(N,h,F)$ is CR-(1,2)-symplectic.\\
Let $\varphi : N \to M$ be an immersion such that $(\varphi^\star h)^{(2,0)}=0$. Then the following assertions are equivalent :
\begin{enumerate}
\item[i)] $\phi$ is pseudo-(1,1)-geodesic.

\item[ii)] $\tilde\varphi : (N,F)\to (Y,\tilde{\mathcal C})$, $x \mapsto d\varphi\Big(\ker(F+i)\Big)$ is a CR map. 
\end{enumerate}
\end{cor}

\begin{cor}
Let $(M,g)$ be an oriented Riemannian $(2n+2)$-manifold and $Y$ the bundle of maximal isotropic subspaces on $T^\C M$  endowed with its canonical CR structure $\tilde{\mathcal C}$. 
Let $(N,h,F)$ be a Riemannian manifold endowed with a skew-adjoint almost $f$-structure $F$ of rank $2n$ such that $(N,h,F)$ is CR-(1,2)-symplectic.
Let $\phi : N \to M$ be an immersion such that $(\phi^\star h)^{(2,0)}=0$. Then the following assertions are equivalent :
\begin{enumerate}
\item[i)] $\phi$ is pseudo-$(1,1)$-geodesic

\item[ii)] $\tilde\phi : (N,F)\to (Y,\tilde{\mathcal C})$ is a CR map. 
\end{enumerate}
\end{cor}

\begin{exm}
Any orientable $3$-dimensional manifold admits a foliation by Riemann surfaces, hence a Levi flat CR structure with $(\nabla F)_{H} =0$, where $H= {\ker}(F^2 + \mathrm{Id})$. Therefore, any $3$-dimensional manifold can be equipped with a CR-(1,2)-symplectic $f$-structure. \\
In particular, by the previous corollary, Eells-Salamon's bijective correspondence between $\phi$  pseudo-harmonic and $\tilde\phi$ CR map, extends to isometric immersions $\phi : (N^3,h)\to (M^4,g)$.
\end{exm}

\begin{exm}
An almost contact manifold $(N,h,F)$ is called nearly cosymplectic if $\nabla F$ is skew-symmetric. This condition implies that the associated $f$-structure is CR-(1,2)-symplectic.\\
If $N$ is a hypersurface of a nearly Kahler manifold $M$, then it inherits an almost contact structure which, under some conditions, is nearly cosymplectic. The inclusion map $\phi$ will then satisfy the conditions of the last corollary.
\end{exm}

\section{Riemannian twistorial structures} 

\indent 
Let $E$ be a Riemannian vector bundle over a manifold $M$. A subbundle $Y\subseteq{\rm Gr}\bigl(E^{\C\!}\bigr)$ such that 
$Y_x\subseteq{\rm Gr}\bigl(E_x^{\C\!}\bigr)$ is an Euclidean $Y_x$-structure on $E_x$\,, for any $x\in M$, is called 
a \emph{Riemannian $Y$-structure on $E$}\,. Suppose, further, that $M$ is Riemannian and there is a morphism of vector bundles 
$\r:E\to TM$ which is an orthogonal projection at each point. If $\r_x$ and $\s_x$ induce a $Y_x$-structure on $T_xM$, for any $x\in M$, 
where $\s:Y\to{\rm Gr}\bigl(E^{\C\!}\bigr)$ is the inclusion, then we say that \emph{$\r$ and $\s$ induce a Riemannian $Y$-structure on $M$}.\\ 
\indent 
Let $E$ be, further, endowed with a \emph{compatible} connection $\nabla$\,; that is, $\nabla$ preserves the 
Riemannian structure of $E$ and induces a connection $\H\subseteq TY$ on $Y$. Then we can construct an almost co-CR structure $\Cal$ on $Y$ 
(that is, $\Cal\subseteq T^{\C\!}Y$ is a complex distribution such that $\Cal+\overline{\Cal}=T^{\C\!}Y$), as follows. 
Firstly, let $\mathcal{B}\subseteq\H^{\C\!}$ be such that $\dif\!\p(\mathcal{B}_y)=\r\bigl(\s(y)\bigr)$\,, for any $y\in Y$, where $\p:Y\to M$ is the projection. 
Then $\Cal=\mathcal{B}\oplus({\rm ker}\dif\!\p)^{0,1}$ is an almost co-CR structure on $Y$.\\ 
\indent 
Let $T=\dif^{\nabla}\!\iota$ be the $\iota\,$-torsion of $\nabla$, where $\iota$ is the inclusion of $TM$ into $E$ as the orthogonal complenet of 
${\rm ker}\r$\,. If $(X,Y,Z)\mapsto\langle\,T(X,Y),Z\,\rangle$ is a section of $\Lambda^3T^*M$, where $\langle\cdot,\cdot\rangle$ is the metric on $E$\,, 
then $T$ is called \emph{totally antisymmetric}.\\ 
\indent 
The following definition specializes notions of \cite{PanWoo-sd}\,, \cite{LouPan-II}\,; compare, also, \cite{Raw-f}\,, \cite{LouPan}\,, \cite{fq_2}\,. 

\begin{defn} \label{defn:Riemannian_twist_str} 
If the $\iota\,$-torsion of $\nabla$ is totally antisymmetric we say that $(E,\r,\s,\nabla)$ is a \emph{Riemannian almost twistorial structure (on $M$)}\,.  
If, also, $\Cal$ is integrable we say that $(E,\r,\s,\nabla)$ is a \emph{Riemannian twistorial structure}. 
\end{defn}  

\indent 
Note that, similarly to \cite{Pan-integrab_co-cr_q}\,, if the automorphism group of the typical fibre of $Y$ acts tranzitively on it, 
the integrability of $\Cal$ is determined by the curvature form and the $\r$-torsion of $\nabla$.\\  
\indent 
Let $M$ be endowed with a Riemannian twistorial structure $(E,\r,\s,\nabla)$\,. Suppose that  
there exists a complex manifold $Z$ endowed with a conjugation and a surjective submersion $\phi:Y\to Z$ intertwining the conjugations 
and such that $\Cal=(\dif\!\phi)^{-1}\bigl(T^{0,1}Z\bigr)$ (we take $\phi$ such that $({\rm ker}\dif\!\phi)^{\C\!}=\Cal\cap\overline{\Cal}$\,). 
We call $Z$ the \emph{twistor space} of $(E,\r,\s,\nabla)$\,. 

\begin{rem} 
Let $M$ be an oriented Riemannian manifold, $\dim M=n$\,. Let $Y^0$ be a compact complex manifold endowed with a conjugation 
and let $G\subseteq{\rm SO}(n)$ be the automorphism group of an Euclidean $Y^0$-structure on $\R^n$ (endowed with its canonical metric). 
If $G$ contains the holonomy group of $M$ then we can build $\p:Y\to M$ and an embedding $\s:Y\to{\rm Gr}\bigl(T^{\C\!}M\bigr)$ 
such that $(TM,{\rm Id}_{TM},\s,\nabla)$ is a Riemannian almost twistorial structure, where $\nabla$ is the Levi-Civita connection of $M$.\\ 
\indent 
Moreover, assuming integrability, the twistor space can be defined at least by restricting to any convex open subset of $M$.  
\end{rem} 

\begin{rem}   
a) Let $(E,\r,\s,\nabla)$ be a Riemannian almost twistorial structure on $M$. Endow $Y$ with the Riemannian metric 
with respect to which $\H=({\rm ker}\dif\!\p)^{\perp}$, $\p$ is a Riemannian submersion and on its fibres the metrics are the same with the ones 
induced from ${\rm Gr}\bigl(E^{\C\!}\bigr)$\,. Then $\Cal$ is coisotropic.\\ 
\indent   
b) Conversely, let $M$ be a Riemannian manifold and let $Z$ be a complex manifold endowed with a conjugation. 
Let $\p$ and $\phi$ be submersions from a Riemannian manifold $Y$ onto $M$ and $Z$, respectively, such that:\\ 
\indent 
\quad(1) $\p$ is a proper Riemannian submersion with totally geodesic fibres;\\ 
\indent 
\quad(2) $\Cal=(\dif\!\phi)^{-1}\bigl(T^{0,1}Z\bigr)$ is a coisotropic distribution on $Y$;\\ 
\indent 
\quad(3) For any $x\in M$, the map $Y_x\to{\rm Gr}\bigl(T_x^{\C\!}M\bigr)$\,, $y\mapsto\dif\!\p\bigl(\Cal_y\bigr)$\,, $(y\in Y_x)$\,, 
is an Euclidean $Y_{x\,}$-structure on $T_xM$, where $Y_x=\p^{-1}(x)$\,;\\ 
\indent 
\quad(4) For any $x\in M$, the inclusion map $Y_x\to Y$ is an isometric embedding, where $Y_x$ is endowed with the K\"ahler metric induced 
from ${\rm Gr}\bigl(T_x^{\C\!}M\bigr)$\,;\\ 
\indent 
\quad(5) For any $x\in M$, the restriction of $\phi$ to $Y_x$ is a holomorphic embedding intertwining the conjugations.\\ 
\indent 
Suppose that $Y$ is induced through a vector bundles morphism $\r:E\to TM$ from a Riemannian vector bundle $E$ 
endowed with a compatible connection $\nabla$ with totally antisymmetric torsion, and $Y\subseteq{\rm Gr}\bigl(E^{\C\!}\bigr)$ 
and $\H=({\rm ker}\dif\!\p)^{\perp}$ associated to $\nabla$. Then $(E,\r,\s,\nabla)$ is a Riemannian twistorial structure, with twistor space $Z$, 
where $\s:Y\to{\rm Gr}\bigl(E^{\C\!}\bigr)$ is the inclusion. For brevity, we say that \emph{$(E,\r,\s,\nabla)$ is given by $\phi$ and $\p$}\,. 
\end{rem} 

\indent 
Let $M$ be endowed with a Riemannian twistorial structure given by $\phi:Y\to Z$ and $\p:Y\to M$. 
Let $\H=({\rm ker}\dif\!\p)^{\perp}$ and $\mathscr{K}=({\rm ker}\dif\!\phi)^{\perp}$ be the horizontal distributions on $Y$ 
determined by $\p$ and $\phi$\,, respectively. Let $J$ be the orthogonal complex structure on $\mathscr{K}$ with respect 
to which $\dif\!\phi|_{\mathscr{K}}$ is complex linear at each point. Note that, $J$ restricts to give the complex structures of the fibres of $\p$\,.\\ 
\indent 
Because $\Cal$ is coisotropic and $({\rm ker}\dif\!\phi)^{\C\!}=\Cal\cap\,\overline{\Cal}$ 
we have $\mathscr{K}^{\C\!}=\Cal^{\perp}\oplus\,\overline{\Cal^{\perp}}$. Together with $\Cal^{\perp}\supseteq({\rm ker}\dif\!\p)^{0,1}$, 
this implies ${\rm ker}\dif\!\p\subseteq\mathscr{K}$; equivalently, ${\rm ker}\dif\!\phi\subseteq\H$. We shall denote by $P_{\H\cap\mathscr{K}}$ 
the orthogonal projection onto $\H\cap\mathscr{K}$\,. 

\begin{thm} \label{thm:Rts} 
Let $M$ be endowed with a Riemannian twistorial structure given by $\phi:Y\to Z$ and $\p:Y\to M$. 
If the fibres of $\phi$ are minimal then the following assertions are equivalent:\\ 
\indent 
\quad{\rm (i)} $\phi$ pulls back holomorphic functions on $Z$ to harmonic functions on $Y$;\\ 
\indent 
\quad{\rm (ii)} $\trace_{\o}\bigl(I^{\H}+P_{\H\cap\mathscr{K}}T\bigr)=0$\,, where $I^{\H}$ is the integrability tensor of $\H$, $\o$ is the K\"ahler form of 
$\bigl(\H\cap\mathscr{K},J|_{\H\cap\mathscr{K}}\bigr)$\,, and $T$ is the section of $\H\otimes\Lambda^2\H^*$ 
corresponding to the torsion of the induced connection on $M$.\\ 
\indent  
Consequently, if $Z$ is K\"ahler, the fibres of $\phi$ are minimal and {\rm (ii)} holds then $\phi$ is a harmonic map. 
\end{thm} 
\begin{proof} 
Let $(E,\r,\s,\nabla)$ be the given Riemannian twistorial structure. Note that, $\r\nabla$ is related to the Levi-Civita connection $\nabla^M$ of $M$ 
by $\r\nabla=\nabla^M+\tfrac12\,T$, where $T$, also, denotes the torsion of $\r\nabla$.\\ 
\indent  
Let $(y_t)_t$ be a curve on $Y$ horizontal with respect to $\p$\,; that is, $\dot{y}_t\in\H_{y_t}$\,, for any $t$\,. Because $\H$ is induced by $\nabla$,  
for any section $s$ of the pull back by $(\p(y_t))_t$ of $E$ such that $s_t\in\s(y_t)$\,, for any $t$\,, we have 
$\nabla_{\dif\!\p(\dot{y}_t)}s\in\s(y_t)$\,, for any $t$\,. Thus, if $s_t\in\s(y_t)\cap T_{\p(y_t)}^{\C\!}M$, for any $t$\,, then 
$(\r\nabla)_{\dif\!\p(\dot{y}_t)}s\in\r\bigl(\s(y_t)\bigr)$\,, for any $t$\,. Together with Lemma \ref{lem:isotropic_part_of_induced}\,, we deduce that 
if $s_t\in\r\bigl(\s(y_t)\bigr)^{\perp}$, for any $t$\,, then $(\r\nabla)_{\dif\!\p(\dot{y}_t)}s\in\r\bigl(\s(y_t)^{\perp}\bigr)$\,, for any $t$\,.\\   
\indent  
As the fibres of $\phi$ are minimal, with notations similar to \cite[Appendix A]{PanWoo-d}\,, assertion (i) is equivalent to $\Div_{\!\mathscr{K}}\!J=0$\,.\\ 
\indent 
Note that, at each $y\in Y$, the $-{\rm i}$ eigenspace of $J$ restricted to $\H_y\cap\mathscr{K}_y$ is the horizontal lift, with respect to $\p$\,, 
of $\r\bigl(\s(y)\bigr)^{\perp}$. Consequently, for any sections $X_1$ and $X_2$ of $\H$ such that $X_2$ is 
a section of ${\rm ker}(J+{\rm i})$\,, the orthogonal projection onto $\H\cap\mathscr{K}$ of $\p^{*\!}(\r\nabla)_{X_1}X_2$ 
is a section of ${\rm ker}(J+{\rm i})$\,.\\   
\indent 
Together with the fact that $\p$ is a Riemannian submersion and its fibres are K\"ahler manifolds, a straightforward calculation gives 
$\Div_{\!\mathscr{K}}\!J=-\tfrac12\trace_{\o}\bigl(I^{\H}+P_{\H\cap\mathscr{K}}T\bigr)$\,.\\ 
\indent 
The last statement follows from \cite{Lou-pseudo_hm} (or \cite[Appendix A]{PanWoo-d}\,). 
\end{proof} 

\begin{cor} \label{cor:Rts} 
Let $M$ be endowed with a Riemannian twistorial structure given by $\phi:Y\to Z$ and $\p:Y\to M$, and such that $E=TM$. 
Then the following assertions are equivalent:\\ 
\indent 
\quad{\rm (i)} $\phi$ pulls back holomorphic functions on $Z$ to harmonic functions on $Y$;\\ 
\indent 
\quad{\rm (ii)} $\trace_{\o}\bigl(I^{\H}+P_{\H\cap\mathscr{K}}T\bigr)=0$\,, where $I^{\H}$ is the integrability tensor of $\H$, $\o$ is the K\"ahler form of 
$\bigl(\H\cap\mathscr{K},J|_{\H\cap\mathscr{K}}\bigr)$ and $T$ is the section of $\H\otimes\Lambda^2\H^*$ 
corresponding to the torsion of the connection.\\ 
\indent  
Consequently, if $Z$ is K\"ahler and {\rm (ii)} holds then $\phi$ is a harmonic map. 
\end{cor} 
\begin{proof} 
As $\p$ is a Riemannian submersion mapping the fibres of $\phi$ to totally geodesic immersions to $M$, the fibres of $\phi$ are totally geodesic, 
and the proof follows from Theorem \ref{thm:Rts}\,.  
\end{proof} 

\begin{rem} \label{rem:Rts}
Conditions (ii) of Theorem \ref{thm:Rts} and Corollary \ref{cor:Rts} are automatically satisfied if $I^{\H}(X,\overline{X})=T(X,\overline{X})=0$\,, 
for any $X\in\Cal^{\perp}\cap\H^{\C\!}$. 
\end{rem} 

\begin{prop} \label{prop:extended_E-S} 
Let $M$ be endowed with a Riemannian twistorial structure given by $\phi:Y\to Z$ and $\p:Y\to M$, and such that $E=TM$ and $\nabla$ is the Levi-Civita 
connection of $M$.\\ 
\indent 
Then, for any complex submanifold $N$ of $Z$ whose preimage by $\phi$ is horizontal with respect to $\p$\,, the immersion 
$\p|_{\phi^{-1}(N)}:\phi^{-1}(N)\to M$ is minimal.  
\end{prop} 
\begin{proof} 
This similar to the proof of Theorem \ref{thm:CR_EelSal} and will be omitted. 
\end{proof} 

\begin{rem} 
With the same notations as in Proposition \ref{prop:extended_E-S}\,, similarly to Corollary \ref{cor:CR_EelSal_1}\,, 
we obtain that $\phi|_{\phi^{-1}(N)}$ (has geodesic fibres and) is PHH, in the sense of \cite{ApApBr} (see \cite{BaiWoo2}\,).  
\end{rem} 

\begin{exm} \label{exm:flat_hts} 
Let $U$ be endowed with an Euclidean $Y$-structure $\s$\,. Then on endowing $U$ with the trivial flat connection we obtain 
a Riemannian twistorial structure whose twistor space is the holomorphic vector bundle $\mathcal{U}$ of the tautological exact sequence \eqref{e:Ets}\,. 
Note that, $\p$ is the canonical projection from $Y\times U$ to $U$ and $\phi$ is the restriction to $Y\times U$ of the canonical projection 
from $Y\times U^{\C\!}$ to $\mathcal{U}$\,.\\ 
\indent 
Furthermore, if $\s$ is a Hermitian $Y$-structure then the radial projection from $U\setminus\{0\}$ onto the unit sphere $S\subseteq U$ induces a 
Riemannian twistorial structure on $S$, with $\nabla$ the restriction to $S$ of the trivial connection on $U$. The corresponding twistor space 
is the complex projectivisation $P\mathcal{U}$ and, note that, the radial projection $U\setminus\{0\}\to S$ is the twistorial map 
(\,\cite{PanWoo-sd}\,, \cite{LouPan-II}\,) corresponding to the holomorphic submersion $\mathcal{U}\setminus\{0\}\to P\mathcal{U}$\,. 
\end{exm} 

\begin{exm} \label{exm:homog_hts} 
Let $G$ be a semisimple complex Lie group and let $P\subseteq G$ be a parabolic subgroup. Let $H\subseteq G$ be a 
closed complex Lie subgroup which is semisimple, contains no normal subgroup of $G$, and $G/H$ is reductive with respect to 
$\mathfrak{m}\subseteq\gf$\,, where $\gf$ is the Lie algebra of $G$.\\ 
\indent 
Suppose that $P$ and $(H,\mathfrak{m})$ satisfy the following compatibility conditions:\\ 
\indent 
\quad(i) $H\cap P$ is a parabolic subgroup of $H$;\\ 
\indent 
\quad(ii) $\mathfrak{m}$ is nondegenerate with respect to the Killing form of $\gf$\,;\\ 
\indent 
\quad(iii) $\mathfrak{p}=\mathfrak{h}\cap\mathfrak{p}+\mathfrak{m}\cap\mathfrak{p}$\,, where $\mathfrak{p}$ and $\mathfrak{h}$ 
are the Lie algebras of $P$ and $H$, respectively;\\  
\indent 
\quad(iv) $\mathfrak{p}^{\perp}=\mathfrak{h}\cap\mathfrak{p}^{\perp}+\mathfrak{m}\cap\mathfrak{p}^{\perp}$ and $\mathfrak{m}\cap\mathfrak{p}^{\perp}$ 
is the orthogonal complement of $\mathfrak{m}\cap\mathfrak{p}$ in $\mathfrak{m}$ (in particular, $\mathfrak{m}\cap\mathfrak{p}$ 
is coisotropic in $\mathfrak{m}$\,);\\ 
\indent 
\quad(v) $H\cap P=\{a\in H\,|\,({\rm Ad}\,a)(\mathfrak{m}\cap\mathfrak{p})=\mathfrak{m}\cap\mathfrak{p}\}$ (note that, `$\subseteq$' 
is automatically satisfied).\\ 
\indent 
We may choose a compact real form $G^{\R}$ of $G$ such that $H^{\R}=G^{\R}\cap H$ is a compact real form of $H$, and 
we assume that $\mathfrak{m}^{\R}=\mathfrak{m}\cap\gf^{\R}$ is complementary to $\mathfrak{h}^{\R}$ in $\gf^{\R}$, where 
$\gf^{\R}$ and $\mathfrak{h}^{\R}$ are the Lie algebras of $G^{\R}$ and $H^{\R}$, respectively.\\ 
\indent 
Endow $M=G^{\R}/H^{\R}$ with the canonical connection $\nabla$ determined by $\mathfrak{m}^{\R}$ 
and with the $G^{\R}$-invariant Riemannian metric given by the restriction to $\mathfrak{m}^{\R}$ of the opposite of the Killing form 
of $\gf^{\R}$. Then the torsion of $\nabla$ is totally antisymmetric. Moreover, $(M,\nabla)$ satisfy (ii) of Corollary \ref{cor:Rts} (this follows from 
the fact that $\mathfrak{p}^{\perp}$ is nilpotent), with $Y=G^{\R}/(H^{\R}\cap P)$, $Z=G/P$, and $\p:Y\to M$ and $\phi:Y\to Z$ 
the canonical projections (but, note that, the Riemannian metric on $Y$ is not necessarily the quotient metric). Furthermore, 
the Riemannian twistorial structure on $M$ is of Hermitian type if and only if $\mathfrak{h}$ contains the reductive part of $\mathfrak{p}$\,; 
that is, $\mathfrak{h}\supseteq\mathfrak{p}\cap\overline{\mathfrak{p}}$ (equivalently, $\mathfrak{m}\cap\mathfrak{p}=\mathfrak{m}\cap\mathfrak{p}^{\perp}$).\\ 
\indent 
A significant particular case is when $\mathfrak{m}=\mathfrak{h}^{\perp}$ in which case (ii)\,, (iii) and (iv)\,, above, are automatically satisfied. 
This contains the case when $\nabla$ is torsion free; equivalently, up to Riemannian covering spaces, $M$ is a Riemannian symmetric space.\\ 
\indent 
Thus, the examples of twistor spaces of \cite{Bry-har_twistor} and \cite{BurRaw} fit into our setting. Also, all of the examples of \cite{fq_2} 
(see in \cite{Pan-hqo} the cases with splitting normal bundle exact sequence for the twistor spheres) whose twistor spaces are generalized flag manifolds, 
also, provide concrete examples of Riemannian twistorial structures.\\ 
\indent 
Moreover, similarly to \cite{Bry-har_twistor} and \cite{BurRaw}\,, if $\mathfrak{m}=\mathfrak{h}^{\perp}$ and $\nabla$ is torsion free, 
we may construct a $G$-invariant holomorphic distribution $\Fa$ on $Z$ whose preimage by $\dif\!\phi$ is horizontal with respect to $\p$ 
and, thus, by Corollary \ref{cor:CR_EelSal_1} and Proposition \ref{prop:extended_E-S}\,, provides a construction of minimal submanifolds of $M$. 
Note that, geometrically, $\Fa$ can be constructed as follows. Firstly, embed $Z$ into a complex projective space through the 
`generalized Pl\"ucker embedding'. Then $\Fa$ is generated by the tangent spaces to the complex projective lines contained in $Z$ 
whose preimages by $\phi$ are horizontal with respect to $\p$ (note that, unless $G=H$, the distribution $\Fa$ is nonzero).\\ 
\indent 
Among the concrete examples (satisfying (ii) of Corollary \ref{cor:Rts}\,) not previously considered are the following:\\ 
\indent 
1a) Take $M=G^{\R}$ with its Riemannian symmetric space structure, $Z=G/P\times G/P$ and $Y=G^{\R}\times G/P$.\\ 
\indent 
1b) Also, with $\nabla$ the $(-)$-connection we obtain another example of Riemannian twistorial structure on $G^{\R}$ with the same twistor space 
but with a different Riemannian structure on $Y$. \\  
\indent 
Note that, if $G^{\R}/H^{\R}$ is a symmetric space with twistor space $G/P$, as above, then the usual totally geodesic 
embedding of $G^{\R}/H^{\R}$ into $G^{\R}$, endowed with the twistorial structure given by (1a)\,, above, is twistorial. This twistorial map 
corresponds to the diagonal (holomorphic) embedding of $G/P$ into $G/P\times G/P$. This way, we have the following:\\ 
\indent 
2) Let $k,p,q\in\mathbb{N}\setminus\{0\}$\,, $k\leq p/2$\,, $M={\rm Gr}^{+}_p(p+q,\R)$ with its Riemannian symmetric space structure, 
$Z$ the Grassmannian of isotropic $k$-dimensional subspaces of $\C^{\!p+q}$ and $Y=\bigl\{(u,v)\in Z\times M\,|\,u\subseteq v^{\C}\bigr\}$. 
Note that, this Riemannian twistorial structure is of Hermitian type if and only if $p=2k$\,, case considered in \cite{Bry-har_twistor} and \cite{BurRaw}\,; 
also, the case $k=1$\,, $p=3$ was considered in \cite{fq_2} (see \cite{Pan-hqo}\,).\\ 
\indent 
3) Let $k\in\mathbb{N}\setminus\{0\}$ and $p_1,\ldots,p_k,n\in\mathbb{N}$ such that $0<p_1<\cdots<p_k<n$\,. 
The flag manifold $Z=F_{p_1,\ldots,p_k}(n,\C\!)$ is a twistor space for 
$M={\rm Gr}_{p_1-p_2+p_3-\cdots}(n,\C)$\,, with $Y$ formed of those pairs $\bigl((\a_1,\ldots,\a_k),\b\bigr)\in Z\times M$ such that 
$(\a_1\cap\b,\a_3\cap\b,\ldots)\in F_{p_1,p_3-p_2,\ldots}(n,\C\!)$ and $(\a_2\cap\b^{\perp},\a_4\cap\b^{\perp},\ldots)\in F_{p_2-p_1,p_4-p_3,\ldots}(n,\C\!)$\,.  
The case $k\leq2$ is covered by \cite{Bry-har_twistor} and \cite{BurRaw}\,, up to the twistoriality of the embeddings of $M$ into ${\rm SU}(n)$\,.\\ 
\indent 
Note that, also, \cite[Example 4.7]{fq_2} (see \cite[Example 4.6]{Pan-hqo}\,) provides an example of a Riemannian reductive homogoneous space 
which is not a symmetric space and which is endowed with a Riemannian twistorial structure, as above.      
\end{exm} 

\begin{thm} \label{thm:sym_hts} 
Any Riemannian symmetric space admits a nontrivial Riemannian twistorial structure invariant under the isometry group. 
\end{thm} 
\begin{proof} 
This follows from Examples \ref{exm:flat_hts} and \ref{exm:homog_hts}\,. 
\end{proof} 

\begin{rem} 
Let $G$ be a semisimple complex Lie group, $P\subseteq G$ a parabolic subgroup and $G^{\R}$ a compact real form of $G$. 
Let $\phi:Y\to Z$ and $\p:Y\to G^{\R}$ be the submersions giving the corresponding 
Riemannian twistorial on $G^{\R}$ as a symmetric space, where $Y=G^{\R}\times G/P$ and $Z=G/P\times G/P$.\\ 
\indent 
Let $a:N\to G^{\R}$ be an immersion. In view of Corollary \ref{cor:CR_EelSal_1} and 
Proposition \ref{prop:extended_E-S} it is useful to understand when $Y$ admits a section over $a$ which is horizontal with respect to the 
connection induced on $Y$ by the the Levi-Civita connnection of $G^{\R}$. A straightforward argument shows that, locally, this condition holds 
if and only if there exists a map $b:N\to G^{\R}$ such that $b^*\theta+{\rm Ad}(b^{-1})A$ takes values in $\mathfrak{p}$\,, 
where $A=\tfrac12\,a^*\theta$\,, with $\theta$ the canonical (left invariant) one-form on $G^{\R}$ and $\mathfrak{p}$ the Lie algebra of $P$.\\ 
\indent 
Conversely, let $A$ be a $\gf^{\R}$-valued one-form on $N$, where $\gf^{\R}$ is the Lie algebra of $G^{\R}$. Suppose that the following conditions 
are satisfied:\\ 
\indent 
\quad(i) $A_x:T_xN\to\gf^{\R}$ is injective, for any $x\in N$\,;\\ 
\indent 
\quad(ii) $\dif\!A+2[A,A]=0$\,;\\ 
\indent 
\quad(iii) Locally, there exists $b:N\to G^{\R}$ such that $b^*\theta+{\rm Ad}(b^{-1})A$ takes values in $\mathfrak{p}$\,.\\ 
\indent 
Then, locally, we can integrate $a^*\theta=2A$ to obtain immersions $a:N\to G^{\R}$ admitting horizontal lifts to $Y$. Finally, 
if the CR structure on $Y$ induces a complex structure on the image of such a lift, from Corollary \ref{cor:CR_EelSal_1}\,, we deduce 
that $a$ is pluriharmonic and the corresponding statement for Proposition \ref{prop:extended_E-S}\,, also, holds. In fact, locally, $\p|_{\phi^{-1}(N)}$ 
is minimal as in Proposition \ref{prop:extended_E-S} (with `$\phi(N)$' replacing `$N$') which shows that the latter approach is more general 
if the CR twistor space is associated to a Riemannian twistorial structure.\\ 
\indent 
The simplest case is when $P/G$ is a Hermitian symmetric space. Then the totally geodesic embeddings of $P/G$ into $G^{\R}$ 
are twistorial pluriharmonic and their restrictions to complex submanifolds of $P/G$ give, as is well known,   
pluriharmonic immersions of uniton number one, in the sense of \cite{OhnVal-plurihar} (cf.\ \cite{Uhl-89}\,), with respect to any 
unitary representation of $G^{\R}$.  
\end{rem} 

\begin{exm} \label{exm:hts_sym_(3)} 
Let $k\in\mathbb{N}\setminus\{0\}$\,, $p_1,\ldots,p_k,n\in\mathbb{N}$ such that $0<p_1<\cdots<p_k<n$\,.  
Let $Z=F_{p_1,\ldots,p_k}(n,\C\!)$ be the twistor space of Example \ref{exm:homog_hts}(3) and $\psi=(\psi_1,\ldots,\psi_k)$ 
a holomorphic immersion to $Z$ such that if $k\geq2$ then 
$\dif\bigl(\G(\psi_j)\bigr)\subseteq\G(\psi_{j+1})$\,, for $j=1,\ldots,{k-1}$\,, 
where $\psi$ also denotes the corresponding flag of holomorphic vector bundles, $\G$ is the functor giving the sheaves of sections,  
and $\dif$ is the covariant derivation of the trivial flat connection on $Z\times\C^{\!n}$.\\ 
\indent 
Denote by $\Pi_{\a}$ the orthogonal projection onto the subspace $\a$ of $\C^{\!n}$ and 
let $\phi_j=\Pi_{\psi_j}-\Pi_{\psi_j^{\perp}}$, 
$j=1,\ldots,k$\,. Then $\phi=\phi_1\cdots\phi_k$ is a twistorial pluriharmonic immersion as in Corollary \ref{cor:CR_EelSal_1}   
(the case $k\leq2$ is well known; compare \cite{Woo-fest} and the references therein), and $\psi$ provides an example 
for Proposition \ref{prop:extended_E-S}\,. Moreover, to deduce that $\phi$ is pluriharmonic it is sufficient to assume $\psi$ holomorphic. 
\end{exm}


\begin{thebibliography}{10} 

\bibitem{ApApBr} 
M.~A.~Aprodu, M.~Aprodu, V.~Br\^inz\u anescu, A class of harmonic submersions and minimal submanifolds, 
\textit{Internat. J. Math.}, {\bf 11} (2000) 1177--1191. 

\bibitem{ArPiSo} 
C.~Arezzo, G.~P.~Pirola, M.~Solci, The Weierstrass representation for pluriminimal submanifolds, 
\textit{Hokkaido Math. J.}, {\bf 33} (2004) 357--367. 

\bibitem{At-57} 
M.~F.~Atiyah, Complex analytic connections in fibre bundles, 
\textit{Trans. Amer. Math. Soc.}, {\bf 85} (1957) 181--207. 

\bibitem{BaiWoo2}
P.~Baird, J.~C.~Wood, \textit{Harmonic morphisms between Riemannian manifolds},
London Math. Soc. Monogr. (N.S.), no. 29, Oxford Univ. Press, Oxford, 2003.  

\bibitem{Bry-har_twistor} 
R.~L.~Bryant, Lie groups and twistor spaces, 
\textit{Duke Math. J.}, {\bf 52} (1985) 223--261. 

\bibitem{BurRaw}
F.~E.~Burstall, J.~H.~Rawnsley, \textit{Twistor theory for Riemannian symmetric spaces.
With applications to harmonic maps of Riemann surfaces},
Lecture Notes in Mathematics, 1424, Springer-Verlag, Berlin, 1990. 

\bibitem{EelSal}
J.~Eells, S.~Salamon, Twistorial construction of harmonic maps of surfaces into four-manifolds,
\textit{Ann. Scuola Norm. Sup. Pisa Cl. Sci. (4)}, {\bf 12} (1985) 589--640.

\bibitem{LeB-CR} 
C.~R.~LeBrun, Twistor CR manifolds and three-dimensional conformal geometry, 
\textit{Trans. Amer. Math. Soc.}, {\bf 284} (1984) 601--616. 

\bibitem{Lou-pseudo_hm} 
E.~Loubeau, Pseudo-harmonic morphisms, 
\textit{Internat. J. Math.}, {\bf 8} (1997) 943--957.

\bibitem{LouPan} 
E.~Loubeau, R.~Pantilie, Harmonic morphisms between Weyl spaces and twistorial maps, 
\textit{Comm. Anal. Geom.}, {\bf 14} (2006) 847--881.

\bibitem{LouPan-II} 
E.~Loubeau, R.~Pantilie, Harmonic morphisms between Weyl spaces and twistorial maps II, 
\textit{Ann. Inst. Fourier (Grenoble)}, {\bf 60} (2010) 433--453. 

\bibitem{fq}
S.~Marchiafava, L.~Ornea, R.~Pantilie, Twistor Theory for CR quaternionic manifolds and related structures,
\textit{Monatsh. Math.}, {\bf 167} (2012) 531--545.  

\bibitem{fq_2}
S.~Marchiafava, R.~Pantilie, Twistor Theory for co-CR quaternionic manifolds and related structures,
\textit{Israel J. Math.}, {\bf 195} (2013) 347--371. 

\bibitem{Mer} 
S.~Merkulov, Existence and geometry of Legendre moduli spaces, 
\textit{Math. Z.}, {\bf 226} (1997) 211--265.

\bibitem{MerSch} 
S.~Merkulov, L.~Schwachh\"ofer, Classification of irreducible holonomies of torsion-free affine connections, 
\textit{Ann. of Math. (2)}, {\bf 150} (1999) 77--149. 

\bibitem{OhnVal-plurihar} 
Y.~Ohnita, G.~Valli, Pluriharmonic maps into compact Lie groups and factorization into unitons, 
\textit{Proc. London Math. Soc. (3)}, {\bf 61} (1990) 546--570.

\bibitem{Pan-tm} 
R.~Pantilie, On a class of twistorial maps, 
\textit{Differential Geom. Appl.}, {\bf 26} (2008) 366--376. 

\bibitem{Pan-integrab_co-cr_q} 
R.~Pantilie, On the integrability of co-CR quaternionic structures, 
\textit{New York J. Math.}, {\bf 22} (2016) 1--20. 

\bibitem{Pan-hqo} 
R.~Pantilie, Quaternionic-like manifolds and homogeneous twistor spaces, 
\textit{Proc. A}, {\bf 472} (2016) 20160598, 11 pp. 

\bibitem{Pan-rtwist}
R.~Pantilie, Twistorial structures revisited, Preprint IMAR, Bucharest, 2016, 
(available from \href{http://arxiv.org/abs/1612.07488}{\tt http://arxiv.org/abs/1612.07488}).  

\bibitem{PanWoo-d} 
R.~Pantilie, J.~C.~Wood, Harmonic morphisms with one-dimensional fibres on Einstein manifolds, 
\textit{Trans. Amer. Math. Soc.}, {\bf 354} (2002) 4229--4243.

\bibitem{PanWoo-sd}
R.~Pantilie, J.~C.~Wood, Twistorial harmonic morphisms with one-dimensional fibres on self-dual four-manifolds,
\textit{Q. J. Math}, {\bf 57} (2006) 105--132. 

\bibitem{Raw-f} 
J.~H.~Rawnsley, $f$-structures, $f$-twistor spaces and harmonic maps, 
\textit{Geometry seminar ``Luigi Bianchi'' II -- 1984}, Lecture Notes in Math., 1164, Springer, Berlin, 1985, 85--159. 

\bibitem{Ro-LeB_nonrealiz}
H.~Rossi, LeBrun's nonrealizability theorem in higher dimensions,
\textit{Duke Math. J.}, {\bf 52} (1985) 457--474. 

\bibitem{Sal-har_holo_maps} 
S.~Salamon, Harmonic and holomorphic maps, 
\textit{Geometry seminar ``Luigi Bianchi'' II -- 1984}, Lecture Notes in Math., 1164, Springer, Berlin, 1985, 161--224.

\bibitem{Som-73} 
A.~J.~Sommese, Borel's fixed point theorem for K\"ahler manifolds and an application, 
\textit{Proc. Amer. Math. Soc.}, {\bf 41} (1973) 51--54. 

\bibitem{Uhl-89} 
K.~Uhlenbeck, Harmonic maps into Lie groups: classical solutions of the chiral model, 
\textit{J. Differential Geom.}, {\bf 30} (1989) 1--50. 

\bibitem{Woo-fest} 
J.~C.~Wood, Explicit constructions of harmonic maps, \textit{Harmonic maps and differential geometry},  
Contemp. Math., 542, Amer. Math. Soc., Providence, RI, 2011, 41--73.


\end{thebibliography}
\end{document}